\documentclass[leqno]{amsart}
\usepackage{amssymb}

\newcommand{\Q}{{\mathbb{Q}}}
\newcommand{\F}{{\mathbb{F}}}

\newcommand{\Z}{{\mathbb{Z}}}

\newcommand{\cE}{{\mathcal{E}}}
\newcommand{\cF}{{\mathcal{F}}}
\newcommand{\cG}{{\mathcal{G}}}
\newcommand{\cI}{{\mathcal{I}}}
\newcommand{\cM}{{\mathcal{M}}}
\newcommand{\cO}{{\mathcal{O}}}
\newcommand{\cS}{{\mathcal{S}}}

\newcommand\Irr{\operatorname{Irr}}
\newcommand\Ind{\operatorname{Ind}}

\newtheorem{thm}{Theorem}[section]
\newtheorem{cor}[thm]{Corollary}
\newtheorem{prop}[thm]{Proposition}
\newtheorem{lem}[thm]{Lemma}
\newtheorem{conj}[thm]{Conjecture}

\theoremstyle{definition}

\newtheorem{rem}[thm]{Remark}

\renewcommand{\leq}{\leqslant}
\renewcommand{\geq}{\geqslant}

\address{M.G.: Department of Mathematical Sciences, King's College, 
Aberdeen AB24 3UE, Scotland, UK}

\email{geck@maths.abdn.ac.uk}

\address{D.H.: Lyc\'ee Descartes, 10 rue des Minimes, 37000 Tours, France}

\email{davidhezard@yahoo.fr}
\begin{document}

\date{}

\title{On the unipotent support of character sheaves}

\author{Meinolf Geck and David H\'ezard}


\begin{abstract} 
Let $G$ be a connected reductive group over $\F_q$, where $q$ is large 
enough and the center of $G$ is connected. We are concerned with 
Lusztig's theory of {\em character sheaves}, a geometric version of 
the classical character theory of the finite group $G(\F_q)$. We show that 
under a certain technical condition, the restriction of a character sheaf 
to its {\em unipotent support} (as defined by Lusztig) is either zero or an 
irreducible local system. As an application, the generalized 
Gelfand-Graev characters are shown to form a $\Z$-basis of the $\Z$-module 
of unipotently supported virtual characters of $G(\F_q)$ (Kawanaka's 
conjecture).
\end{abstract}

\maketitle

\begin{center}
{\it Dedicated to Professors Ken-ichi Shinoda and Toshiaki Shoji  on their 
60th birthday}
\end{center}
\bigskip

\pagestyle{myheadings}
\markboth{Geck and H\'ezard}{Unipotent support of character sheaves}

%%%%%%%%%%%%%%%%%%%%%%%%%%%%%%%%%%%%%%%%%%%%%%%%%%%%%%%%%%%%%%%%%%%%%%%%%%%
\section{Introduction} \label{sec1}

Let $G$ be a connected reductive algebraic group over $\overline{\F}_p$, an 
algebraic closure of the finite field with $p$ elements where $p$ is a prime. 
Let $q$ be a power of $p$ and assume that $G$ is defined over the finite
field $\F_q\subseteq \overline{\F}_p$, with corresponding Frobenius map 
$F \colon G \rightarrow G$. Then it is an important problem to determine
and to understand the values of the irreducible characters (in the sense 
of Frobenius) of the finite group $G^F$. For this purpose, Lusztig 
\cite{L4} has developed the theory of {\em character sheaves}; see \cite{Lg}
for a general overview. This theory produces some geometric objects over 
$G$ (provided by intersection cohomology with coefficients in 
$\overline{\Q}_\ell$, where $\ell\neq p$ is a prime) from which the 
irreducible characters of $G^F$ can be deduced for any $q$. In this way, 
the rather complicated patterns involved in the values of the irreducible 
characters of $G^F$ are seen to be governed by geometric principles. 

In this paper, we discuss an example of this interrelation between geometric
principles and properties of character values. On the geometric side, we will
be concerned with the restriction of a character sheaf $A$ to the unipotent 
variety of $G$. Under some restriction on $p$, Lusztig \cite{LuUS} has 
associated to $A$ a well-defined unipotent class $\cO_A$ of $G$, called its 
{\em unipotent support}. We will be interested in the restriction of $A$ to 
$\cO_A$. Under a certain technical condition (formulated in \cite{G1}, 
following a suggestion of Lusztig) the restriction of $A$ to $\cO_A$ is 
either zero or an irreducible $G$-equivariant local system on $\cO_A$ (up to
shift); see Section~3. The verification of that technical condition can be 
reduced to a purely combinatorial problem, involving the induction of 
characters of Weyl groups, the Springer correspondence and the data on 
families of characters in Chapter~4 of Lusztig's book \cite{L3}. The details 
of the somewhat lengthy case-by-case verification are worked out in the 
second author's thesis \cite{david}; the main ingredients will be explained 
in Section~2.

On the character-theoretic side, we will consider the {\em generalized
Gelfand--Graev representations} (GGGR's for short) introduced by Kawanaka 
\cite{K1}, \cite{K2}. In Section~4, assuming that $p,q$ are large and the 
center of $G$ is connected, we deduce that Kawanaka's conjecture \cite{K3} 
holds, that is, the characters of the various GGGR's of $G^F$ form a 
$\Z$-basis of the $\Z$-module of unipotently supported virtual characters 
of $G^F$. As a further application, in Proposition~\ref{Charac}, we obtain 
a new characterisation of GGGR's in terms of vanishing properties of their 
character values.

%%%%%%%%%%%%%%%%%%%%%%%%%%%%%%%%%%%%%%%%%%%%%%%%%%%%%%%%%%%%%%%%%%%%%%%%%%%
\section{The Springer correspondence, families and induction} 
\label{sec2}

In this section, we deal with the combinatorial basis for the discussion
of the unipotent support of character sheaves. We keep the basic assumptions 
of the introduction: $G$ is a connected reductive algebraic group over 
$\overline{\F}_p$; we assume throughout that $p$ is a good prime for $G$ 
and that the center of $G$ is connected. Let $B\subseteq G$ be a Borel 
subgroup and $T\subseteq B$ a maximal torus. Let 
$W=\mbox{N}_G(T)/T$ be the Weyl group of $G$, with set of generators $S$ 
determined by the choice of $T\subseteq B$. 

Let $\Irr(W)$ be the set of irreducible characters of $W$ (over an
algebraically closed field of characteristic $0$). The Springer 
correspondence associates with each $E \in \Irr(W)$ a pair $(u,\psi)$ 
where $u \in G$ is unipotent (up to $G$-conjugacy) and $\psi$ is an 
irreducible character of the group of components $\operatorname{A}_G(u)=
\operatorname{C}_G(u)/\operatorname{C}_G(u)^\circ$; see \cite[\S 13.1]{L3}. 
We write this correspondence as $E \leftrightarrow (u,\psi)$.

Now we can define three invariants $a_E$, $b_E$ and $d_E$ for $E \in \Irr(W)$. 

\begin{itemize}
\item[$b_E$] is the smallest $i\geq 0$ such that $E$ appears with non-zero
multiplicity in the $i$th symmetric power of the reflection representation 
of $W$; see \cite[(4.1.2)]{L3}.
\item[$a_E$] is the largest $i\geq 0$ such that $u^i$ divides the generic
degree $D_E(u) \in {\Q}[u]$ defined in terms of the generic Iwahori--Hecke
algebra over ${\Q}[u^{1/2},u^{-1/2}]$; see \cite[(4.1.1)]{L3}.
\item[$d_E$] is $\dim \mathfrak{B}_u$ where $\mathfrak{B}_u$ is the variety of
Borel subgroups containing a unipotent $u \in G$ such that $E
\leftrightarrow (u,\psi)$ for some $\psi\in \Irr(\mbox{A}_G(u))$; see 
\cite[\S 13.1]{L3}.
\end{itemize}
We will be interested in several compatibility properties of these 
invariants.

\begin{lem} \label{lem21} We have $a_E\leq d_E \leq b_E$ for all $E \in 
\Irr(W)$.
\end{lem}

\begin{proof} See \cite[Cor.~10.9]{LuUS} for the first inequality and
\cite[\S 1.1]{Sp} for the second. The inequality $a_E\leq b_E$ was first
observed by Lusztig; see \cite[4.1.3]{L3}.
\end{proof}

Recall that $\Irr(W)$ is partitioned into {\em families} and that each
family contains a unique {\em special} $E \in \Irr(W)$, that is, a character
such that $a_E=b_E$; see \cite[4.1.4]{L3}. Furthermore, in \cite[Chap.~4]{L3}, 
Lusztig associates with any family $\cF\subseteq \Irr(W)$ a finite group 
$\cG_{\cF}$, case-by-case for each type of finite Weyl group. (The groups 
$\cG_{\cF}$ form a crucial ingredient in the statement of the Main 
Theorem~4.23 of \cite{L3}.) If $G$ is simple modulo its center, then 
$\cG_{\cF}\cong$ $\mathfrak{S}_3$, $\mathfrak{S}_4$, $\mathfrak{S}_5$ 
or $(\Z/2\Z)^e$ for some $e \geq 0$. 

Now let $G^*$ be the Langlands dual of $G$, with Borel subgroup $B^*$ and 
maximal torus $T^*\subseteq G^*$. Let $W^*=\mbox{N}_{G^*}(T^*)/T^*$ be the 
Weyl group of $G^*$, with generating set $S^*$ determined by $T^*\subseteq 
B^*$. We can naturally identify $W$ and $W^*$. Note that $a_E$ and $b_E$ are
independent of whether we regard $E$ as a representation of $W$ or of $W^*$.
However, it does make a difference as far as $d_E$ is concerned.

Let $s \in G^*$ be semisimple and $W_s$ be the Weyl group of $\mbox{C}_{G^*}
(s)$. (Note that $\mbox{C}_{G^*}(s)$ is a connected reductive group since 
the center of $G$ is connected.) Replacing $s$ by a conjugate, we may assume
that $s\in T^*$. Then $W_s$ is a subgroup of $W^*$ and, hence, may be 
identified with a subgroup of $W$. So we can consider the induction of 
characters from $W_s$ to $W$.

\begin{prop} \label{prop22} Let $s \in G^*$ be semisimple and $\cF 
\subseteq \Irr(W_s)$ be a family. If $E_0$ is the special character 
in $\cF$, then we have
\[ \Ind_{W_s}^W(E_0)=E_0'\,+\,\mbox{a combination of $\tilde{E} \in 
\Irr(W)$ with $b_{\tilde{E}}>d_{\tilde{E}}\geq b_{E_0}$},\]
where $E_0' \in \Irr(W)$ is such that $b_{E_0'}=d_{E_0'}=b_{E_0}$; 
furthermore, $E_0'\leftrightarrow (u,1)$ under the Springer correspondence, 
where $1$ stands for the trivial character.
\end{prop}

\begin{proof} See \cite[\S 10]{LuUS} and \cite[\S 13.1]{L3}.
\end{proof}

We are now looking for a condition which guarantees that all $\tilde{E}
\neq E_0'$ occurring in the decomposition of $\Ind_{W_s}^W(E_0)$ have 
$d_{\tilde{E}}>b_{E_0}$. Following a suggestion of Lusztig, such a condition 
has been formulated in \cite[4.4]{G1}. In order to state it, we introduce 
the following notation.

Let $\cS_G$ be the set of all pairs $(s,\cF)$ where $s \in G^*$ is semisimple 
(up to $G^*$-conjugacy) and $\cF \subseteq \Irr(W_s)$ is a family. Following 
\cite[\S 13.3]{L3}, we define a map 
\[ \Phi_G \colon \cS_G \rightarrow \{\mbox{unipotent classes of $G$}\},\]
as follows. Let $(s,\cF) \in \cS_G$ and $E_0\in \cF$ be special. Then
consider the induction $\Ind_{W_s}^W(E_0)$ and let $E_0'$ be as in 
Proposition~\ref{prop22}. Now define $\cO=\Phi_G(s,\cF)$ to be the 
unipotent class containing $u$ where $E_0'\leftrightarrow (u,1)$ under 
the Springer correspondence. 

\begin{prop}[H\'ezard \protect{\cite{david}}] \label{prop24} Assume that 
$s \in G^*$ is semisimple and isolated, that is, $\operatorname{C}_{G^*}(s)$ 
is not contained in a Levi complement of any proper parabolic subgroup of 
$G^*$. Let $\cF\subseteq \Irr(W_s)$ be a family and assume that  
\begin{equation*}
\mbox{\fbox{$|\cG_{s,\cF}|=|\operatorname{A}_G(u)|
\quad\mbox{where}\quad u\in\cO=\Phi_G(s,\cF).$}} \tag{$*$}
\end{equation*}
Then the following sharper version of Proposition~\ref{prop22} holds: If
$E_0$ is the special character in $\cF$, then we have
\[ \Ind_{W_s}^W(E_0)=E_0'\,+\,\mbox{a combination of $\tilde{E} \in 
\Irr(W)$ with $d_{\tilde{E}}>b_{E_0}$}.\]
\end{prop}

\begin{proof} In the setting of Proposition~\ref{prop22}, let us write 
\[ \Ind_{W_s}^W(E_0)=E_0'+E_0''\,+\,\mbox{a combination of $\tilde{E} \in 
\Irr(W)$ with $d_{\tilde{E}}>b_{E_0}$}\]
where $E_0''$ is the sum of all $\tilde{E}\in \Irr(W)$ such that 
$d_{\tilde{E}}=b_{E_0}$, $\tilde{E}\neq E_0'$ and $\tilde{E}$ appears 
in $\Ind_{W_s}^W(E_0)$. Thus, we must show that $E_0''=0$ if ($*$) holds. 
By standard arguments, this can be reduced to the case where $G$ is simple 
modulo its center.  

The reflection subgroups of $W$ which can possibly arise as $W_s$ for 
some semisimple element $s \in G^*$ are classified by a standard algorithm; 
see \cite{der}. 

Now, if $G$ is of exceptional type, $E_0''$ can be computed in all cases 
using explicit tables for the Springer correspondence \cite{Sp} and 
induce/restrict matrices for the characters of Weyl groups; see 
\cite[\S 2.6]{david} where tables specifying $E_0''$ can be found 
for each type of $G$. By inspection of these tables, one checks that 
if ($*$) holds, then $E_0''=0$. 

If $G$ is of classical type, the induction of characters of Weyl
groups and the Springer correspondence can be described in purely
combinatorial terms, involving manipulations with various kinds
of symbols (\cite[\S 13]{L3a}). The condition ($*$) can also be 
formulated in purely combinatorial terms. Using this information, 
it is then possible to check that, if ($*$) holds, then $E_0''=0$. 
For the details of this verification, see \cite[Chap.~3]{david}.

We remark that, for $G$ of type $B_n$, Lusztig \cite[4.10]{L5} has shown 
that $E_0''=0$ even without assuming that ($*$) holds.
\end{proof}

Finally, the following result settles the question of when condition
($*$) is actually satisfied. 

\begin{prop}[Lusztig \protect{\cite[13.3, 13.4]{L3}}\footnote{Note added
January 2008: A new recent preprint by Lusztig \cite{LuNeu} provides a 
detailed proof of the statements in \cite[13.3, 13.4]{L3}.}; see also 
H\'ezard \protect{\cite{david}}] \label{prop25} Let $\cO$ be a unipotent class.
Then
\[ |\cG_{s,\cF}|\leq |\operatorname{A}_G(u)| \quad \mbox{for all $(s,\cF)
\in \cS_G$ such that $u \in \cO=\Phi_G(s,\cF)$}.\]
Furthermore, there exists some $(s,\cF)$ where $s$ is isolated and we 
have equality. If $\cO$ is $F$-stable (where $F$ is a Frobenius map on 
$G$), then such a pair $(s,\cF)$ can be chosen to be $F$-stable, too.
\end{prop}

\begin{proof} Again, this can be reduced to the case where $G$ is simple 
modulo its center, where the assertion is checked case-by-case along 
the lines of the proof of Proposition~\ref{prop24}.  The existence of
suitable semisimple elements $s\in G^*$ with centralisers of the required
type is checked using the tables in \cite{der2}, \cite{der} (for $G$ of 
exceptional type) or using explicit computations with suitable 
matrix representations (for $G$ of classical type). Again, see 
\cite{david} for more details.
\end{proof}

It would be interesting to find proofs of Propositions~\ref{prop24}  and
\ref{prop25} which do not rely on a case-by-case argument.

%%%%%%%%%%%%%%%%%%%%%%%%%%%%%%%%%%%%%%%%%%%%%%%%%%%%%%%%%%%%%%%%%%%%%%%%%%%
\section{Unipotent support} \label{sec3}

Recall that $G$ is assumed to have a connected center and that we are 
working over a field of good characteristic. Now let $\hat{G}$ be the 
set of character sheaves on $G$ (up to isomorphism) over 
$\overline{\Q}_\ell$ where $\ell$ is a prime, $\ell \neq p$. By Lusztig 
\cite[\S 17]{L4}, we have a natural partition 
\[ \hat{G}=\coprod_{(s,\cF)\in \cS_G} \hat{G}_{s,\cF} \qquad\mbox{where}
\qquad \hat{G}_{s,\cF} \stackrel{1{-}1}{\longleftrightarrow} 
\cM(\cG_{\cF}).\]
Here, as in Section~2, $\cG_\cF$ is the finite group associated to a family
$\cF \subseteq \Irr(W_s)$ as in \cite[Chap.~4]{L3}. Furthermore, for 
any finite group $\Gamma$, the set $\cM(\Gamma)$ consists of all pairs 
$(x,\sigma)$ (up to conjugacy) where $x \in \Gamma$ and $\sigma \in 
\Irr(\mbox{C}_{\Gamma}(x))$.

Also recall that we have a natural map $\Phi_G \colon \cS_G\rightarrow 
\{\mbox{unipotent classes of $G$}\}$, defined as in \cite[\S 3.3]{L3}. From 
now on, we assume that $p$ is large enough, so that the main results of 
Lusztig \cite{LuUS} hold. (Here, ``large enough'' means that we can operate 
with the Lie algebra of $G$ as if we were in characteristic~$0$, e.g., we 
can use $\exp$ to define a morphism from the nilpotent variety in the Lie 
algebra to the unipotent variety of $G$.)

\begin{thm}[Lusztig \protect{\cite[Theorem~10.7]{LuUS}}] \label{thm31}
Let $(s,\cF) \in \cS_G$ and $\cO=\Phi_G(s,\cF)$ be the associated unipotent 
class. Then the following hold.
\begin{itemize}
\item[(a)] There exists some $A \in \hat{G}_{s,\cF}$ and an element
$g \in G$ with Jordan decomposition $g=g_sg_u=g_sg_u$ (where $g_s$ is
semisimple and $g_u \in \cO$) such that $A|_{\{g\}}\neq 0$.
\item[(b)] For any A$ \in \hat{G}_{s,\cF}$, any unipotent class $\cO'\neq 
\cO$ with $\dim \cO'\geq \dim \cO$, and any $g'\in G$ with unipotent part 
in $\cO'$, we have $A|_{\{g'\}}=0$.
\end{itemize}
\end{thm}

Consequently, the class $\cO$ is called the {\em unipotent support} for
the character sheaves in $\hat{G}_{s,\cF}$. Note that it may actually 
happen that $A|_{\cO}=0$ for $A \in \hat{G}_{s,\cF}$. 

Given a unipotent class $\cO$, we denote by $\cI_{\cO}$ the set of 
irreducible $G$-equivariant $\overline{\Q}_\ell$-local systems on 
$\cO$ (up to isomorphism).

\begin{thm}[Geck \protect{\cite[Theorem~4.5]{G1}}; see  also the remarks 
in Lusztig \protect{\cite[1.6]{L5}}] \label{finmor} Let $s\in G^*$ be 
semisimple and $\cF\subseteq \Irr(W_s)$ be a family. Let $\cO=
\Phi_G(s,\cF)$ be the associated unipotent class and assume that 
condition {\rm ($*$)} in Proposition~\ref{prop24} is satisfied. 
Then, for any $A \in \hat{G}_{s,\cF}$, the restriction $A|_{\cO}$ is 
either zero or an irreducible $G$-equivariant local system (up to shift). 
Furthermore, the map $A \mapsto A|_{\cO}$ defines a bijection from the set
of all $A\in \hat{G}_{s,\cF}$ with $A|_{\cO}\neq 0$ onto $\cI_{\cO}$. 
\end{thm}

(Note: In \cite[Theorem~4.5]{G1}, the conclusion of Proposition~\ref{prop24},
i.e., the validity of the sharper version of Proposition~\ref{prop22}, 
was added as an additional hypothesis; this can now be omitted.)

Now let $q$ be a power of $p$ and assume that $G$ is defined over
$\F_q\subseteq \overline{\F}_p$, with corresponding Frobenius map 
$F \colon G \rightarrow G$. We translate the above results to class
functions on the finite group $G^F$.

If $A$ is a character sheaf on $G$ then 
its inverse image $F^*A$ under $F$ is again a character sheaf. There 
are only finitely many $A$ such that $F^*A$ is isomorphic to $A$; such
a character sheaf will be called $F$-stable. Let $\hat{G}^F$ be the set
of $F$-stable character sheaves. For any $A\in \hat{G}^F$ we choose an 
isomorphism $\phi \colon F^*A \stackrel{\sim}{\rightarrow} A$ and we form 
the characteristic function $\chi_{A,\phi}$. This is a class function 
$G^F\rightarrow \overline{\Q}_\ell$ whose value at $g$ is the alternating 
sum of traces of $\phi$ on the stalks at $g$ of the cohomology sheaves 
of $A$. Now $\phi$ is unique up to scalar hence $\chi_{A,\phi}$ is unique 
up to scalar. Lusztig \cite[\S 25]{L4} has shown that 

\begin{center}
{\em $\{\chi_{A,\phi}\mid A\in \hat{G}^F\}$ is a basis of the vector 
space of class functions $G^F\rightarrow \overline{\Q}_\ell$.}
\end{center}
Let $\cO$ be an $F$-stable unipotent class of $G$. We denote by 
$\cI_{\cO}^F$ the set of all $\cE\in \cI_\cO$ such that $\cE$ is 
isomorphic to its inverse image $F^*\cE$ under $F$. For any such $\cE$, 
we can define a class function $Y_\cE\colon G^F\rightarrow
\overline{\Q}_\ell$ as in \cite[(24.2.2)--(24.2.4)]{L4}. We have 
$Y_\cE(g)=0$ for $g \not\in\cO^F$ and $Y_{\cE}(g)=\mbox{Trace}(\psi,\cE_g)$
for $g \in \cO^F$, where $\psi\colon F^*\cE\stackrel{\sim}{\rightarrow}\cE$ 
is a suitably chosen isomorphism. On the level of characteristic functions, 
Theorem~\ref{finmor} translates to the following statement (see 
\cite[\S 2,\S3]{L5}, where such a translation is discussed in a more 
general setting):

\begin{cor} \label{cor33} Let $(s,\cF)\in \cS_G$ be $F$-stable and 
$\cO=\Phi_G(s,\cF)$ be the associated unipotent class (which is 
$F$-stable). Assume that condition {\rm ($*$)} in 
Proposition~\ref{prop24} holds. Then, for any $F$-stable $A \in
\hat{G}_{s,\cF}$, we have either $\chi_{A,\phi}(g)=0$ for all 
$g \in \cO^F$ or $\phi$ can be normalized such that $\chi_{A,\phi}(g)
=Y_{\cE}(g)$ for all $g \in \cO^F$ where $\cE=A|_{\cO}$.
\end{cor}

Now let us consider the irreducible characters of $G^F$.
Lusztig \cite{L3} has shown that we have a natural partition
\[\Irr(G^F)=\coprod_{(s,\cF)\in \cS_G^F}\Irr_{s,\cF}(G^F).\]
Furthermore, each piece $\Irr_{s,\cF}(G^F)$ in this partition is
parametrized by a ``twisted'' version of the set $\cM(\cG_{\cF})$; see
\cite[Chap.~4]{L3}. Lusztig \cite{L4} gave a precise conjecture about the
expression of the characteristic functions of $F$-stable character sheaves
as linear combinations of the irreducible characters of $G^F$. Since
we are assuming that $G$ has a connected center (and $p$ is large), this
conjecture is known to hold by Shoji \cite{S2}. In particular, the following 
statement holds:

\begin{prop}[Shoji \protect{\cite{S2}}] \label{shoji} Let 
$(s,\cF)\in \cS_G^F$ and $A \in \hat{G}_{s,\cF}$ be $F$-stable. Then 
$\chi_{A,\phi}$ is a linear combination of the irreducible 
characters in $\Irr_{s,\cF}(G^F)$.
\end{prop}

We can now deduce the following result, whose statement only involves
the values of the irreducible characters of $G^F$, but whose proof relies
in an essential way on the above results on character sheaves.

\begin{cor} \label{cor35} Let $\cO$ be an $F$-stable unipotent class
and $u_1,\ldots,u_d$ be representatives for the $G^F$-conjugacy classes
contained in $\cO$. Let $(s,\cF)\in \cS_G$ be $F$-stable such that $\cO=
\Phi_G(s,\cF)$ and condition {\rm ($*$)} in Proposition~\ref{prop24} 
holds. Then there exist $\rho_1,\ldots, \rho_d \in \Irr_{s,\cF}(G^F)$ 
such that the matrix $\bigl(\rho_i(u_j) \bigr)_{1 \leq i,j\leq d}$ has a 
non-zero determinant.
\end{cor}

\begin{proof} By the proof of \cite[24.2.7]{L4}, there are precisely $d$
irreducible $G$-equivariant local systems $\cE_1,\ldots,\cE_d$ on $\cO$ 
(up to isomorphism) which are isomorphic to their inverse image under $F$; 
furthermore, the matrix $(Y_{\cE_i}(u_j))_{1\leq i,j\leq d}$
is non-singular. 

By Theorem~\ref{finmor}, we can find $A_1,\ldots,A_d \in \hat{G}_{s,\cF}$ 
such that $A_i|_{\cO}=\cE_i$ for all $i$. Since each $\cE_i$ is isomorphic 
to its inverse image under $F$, the same is true for $A_i$ as well. 
(Indeed, since $(s,\cF)$ is $F$-stable, we have $F^*A_i\in\hat{G}_{s,\cF}$
for all $i$; furthermore, $F^*A_i|_{\cO} \cong F^*\cE_i \cong \cE_i$. So 
we must have $F^*A_i \cong A_i$ by Theorem~\ref{finmor}.) By 
Corollary~\ref{cor33}, we have $\chi_{A_i,\phi_i}=Y_{\cE_i}$ for all 
$i$ (where $\phi_i$ is normalized suitably). It follows that the matrix 
$\bigl(\chi_{A_i, \phi_i}(u_j) \bigr)_{1 \leq i,j\leq d}$ has a 
non-zero determinant.

By Proposition~\ref{shoji}, every $\chi_{A_i,\phi_i}$ can be expressed 
as a linear combination of the characters in $\Irr_{s,\cF}(G^F)$. Hence 
there must exist $\rho_1,\ldots,\rho_d \in \Irr_{s,\cF}(G^F)$ such that 
the matrix $\bigl(\rho_i(u_j)\bigr)_{1 \leq i,j\leq d}$ has a non-zero 
determinant.
\end{proof}

%%%%%%%%%%%%%%%%%%%%%%%%%%%%%%%%%%%%%%%%%%%%%%%%%%%%%%%%%%%%%%%%%%%%%%%%%%%
\section{Kawanaka's conjecture} \label{sec4}

Kawanaka \cite{K2} has shown that, assuming we are in good characteristic,
one can associate with every unipotent element $u \in G^F$ a so-called 
{\em generalized Gelfand--Graev representation} $\Gamma_u$ (GGGR for short). 
They are obtained by inducing certain irreducible representations from 
unipotent radicals of parabolic subgroups of $G^F$. At the extreme cases 
when $u$ is trivial or a regular unipotent element we obtain the regular 
representation of $G^F$ or an ordinary Gelfand--Graev representation, 
respectively. Subsequently, assuming that $p,q$ are large, Lusztig 
\cite{LuUS} gave a geometric interpretation of GGGR's in the framework of 
the theory of character sheaves. 

\begin{conj}[Kawanaka \protect{\cite[(3.3.1)]{K1}}] \label{Kconj} The 
characters of the various {\rm GGGR}'s of $G^F$ form a $\Z$-basis of the 
$\Z$-module of unipotently supported virtual characters of $G^F$.
\end{conj}

By Kawanaka \cite[Theorem~2.4.3]{K3}, the conjecture holds if the center
of $G$ is connected and $G$ is of type $A_n$ or of exceptional type. In
this section, assuming that $p,q$ are large enough, we will show that 
it also holds for $G$ of classical type.

Given a unipotent element $u \in G^F$, denote by $\gamma_u$ the character
of the GGGR $\Gamma_u$. The usual hermitian scalar product for class 
functions on $G^F$ will be denoted by $\langle\; ,\; \rangle$. The 
following (easy) result provides an effective method for verifying that 
the above conjecture holds.

\begin{lem} \label{triang1}
Let $u_1,\ldots,u_n$ be representatives for the conjugacy classes of
unipotent elements in $G^F$. Assume that there exist virtual characters 
$\rho_1,\ldots,\rho_n$ of $G^F$ such that the matrix of scalar products 
$(\langle \rho_i,\gamma_{u_j} \rangle)_{1 \leq i,j \leq n}$ is invertible 
over $\Z$. Then Conjecture~\ref{Kconj} holds. 
\end{lem}

\begin{proof} Since the above matrix of scalar products is invertible,
$\gamma_{u_1},\ldots,\gamma_{u_n}$ are linearly independent class functions 
on $G^F$. Consequently, they form a basis of the 
$\overline{\Q}_\ell$-vectorspace of unipotently supported class functions 
on $G^F$. In particular, given any unipotently supported virtual character 
$\chi$ of $G^F$, we can write $\chi=\sum_{i=1}^n a_j \gamma_j$ where $a_j \in 
\overline{\Q}_\ell$, and it remains to show that $a_j \in \Z$ for all~$j$. 

To see this, consider the scalar products of $\chi$ with the virtual 
characters $\rho_i$. We obtain $\sum_{j} a_j \langle \rho_i, \gamma_j 
\rangle=\langle \rho_i,\chi\rangle \in \Z$ for all $i=1,\ldots,n$.  Since 
the matrix of scalar products $(\langle \rho_i,\gamma_j \rangle)$ is 
invertible over $\Z$, we can invert these equations and conclude that $a_j 
\in \Z$ for all~$j$, as desired. 
\end{proof}

%\begin{rem} \label{rem31} Assume that the center of $G$ is connected and 
%$G$ is of type $A_n$ or of exceptional type $G_2$, $F_4$, $E_6$, $E_7$ or 
%$E_8$. In these cases, Kawanaka \cite{K1}, \cite{K2} obtained explicit 
%formulas for the values of the characters of the GGGR's. Combined with 
%Lusztig's classification of characters \cite{L3}, this yields explicit 
%formulas for the scalar products of the characters of the GGGR's with all 
%irreducible characters of $G^F$. (Actually, in \cite{K2}, it assumed that
%$p$ is large. But this hypothesis is only used in reference to results on 
%the Green functions of $G^F$, which are now known to hold in general; 
%see Shoji \cite{S2}.) From these formulas, it is easily verified that 
%irreducible characters $\rho_1, \ldots, \rho_r$ of $G^F$ with properties 
%as required in Lemma~\ref{triang1} exist. Hence, Conjecture~\ref{Kconj} 
%holds in these cases, as pointed out by Kawanaka \cite[Theorem~2.4.3]{K3}.
%In type $A_n$, each $\rho_i$ can even be taken to be a unipotent character. 
%Thus, in order to prove Conjecture~\ref{Kconj} (for $G$ with a connected
%center), it remains to consider a group $G$ of classical type $B_n$, $C_n$ 
%or $D_n$. The difficulty in these cases arises from the fact that one also 
%has to take into account non-unipotent characters. 
%\end{rem}

Let $\mbox{D}_G$ be the Alvis--Curtis--Kawanaka duality operation on 
the character ring of~$G^F$. For any $\rho \in \Irr(G^F)$, there is a sign
$\varepsilon_\rho=\{\pm 1\}$ such that 
\[\rho^*:=\varepsilon_\rho \mbox{D}_G(\rho) \in \Irr(G^F).\] 
The following result will be crucial for dealing with groups of classical 
type. We assume from now on that the center of $G$ is connected and that 
$p,q$ are large, so that the results in Section~3 can be applied. 

\begin{prop} \label{prop41} Let $\cO$ be an $F$-stable unipotent class and 
$u_1,\ldots,u_d$ be representatives for the $G^F$-conjugacy classes 
contained in $\cO$. Let $(s,\cF)\in \cS_G$ be $F$-stable such that $\cO=
\Phi_G(s, \cF)$ and condition {\rm ($*$)} in Proposition~\ref{prop24} holds. 

Assume that $\cG_\cF$ is abelian. Then there exist $\rho_1,\ldots, \rho_d 
\in \Irr_{s,\cF}(G^F)$ such that $\langle \rho_i^*,\gamma_{u_j}\rangle=
\delta_{ij}$ for $1\leq i,j\leq d$.
\end{prop}

\begin{proof} The following argument is inspired by the proof of 
\cite[Proposition~5.6]{G3}. By \cite[Theorem~11.2]{LuUS} and the discussion 
in \cite[Remark~3.8]{gema}, we have
\[ \sum_{i=1}^d [\mbox{A}_G(u_i):\mbox{A}_G(u_i)^F]\,\langle \rho^*, 
\gamma_{u_i} \rangle=\frac{|\mbox{A}_G(u_1)|}{n_\rho} \quad \mbox{for 
any $\rho \in \Irr_{s,\cF}(G^F)$},\]
where $n_\rho\geq 1$ is an integer determined as follows; see 
\cite[4.26.3]{L3}. Let $E_0\in \Irr(W_s)$ be the special character 
in $\cF$. Then 
\[ \rho(1)=\pm n_\rho^{-1}\, q^{a_{E_0}} N \qquad \mbox{where $N$ is 
an integer, $N\equiv 1 \bmod q$};\]
note also that $n_\rho$ is divisible by bad primes only.

Now, Lusztig \cite[4.26.3]{L3} actually gives a precise formula for
the integer $n_\rho$, in terms of a certain Fourier coefficient. In the
case where $\cG_\cF$ is abelian, this Fourier coefficient evaluates to 
$|\cG_{\cF}|^{-1}$. Thus, we have $n_\rho=|\cG_\cF|^{-1}$. So, since ($*$)
is assumed to hold, we obtain 
\[ \sum_{i=1}^d [\mbox{A}_G(u_i):\mbox{A}_G(u_i)^F]\,\langle \rho^*,
\gamma_{u_i} \rangle=1\quad \mbox{for any $\rho \in \Irr_{s,\cF}(G^F)$}.\]
Now note that each term $[\mbox{A}_G(u_i):\mbox{A}_G(u_i)^F]$ is a 
positive integer and each term $\langle \rho^*, \gamma_{u_i} \rangle$ is 
a non-negative integer. It follows that, given 
$\rho\in \Irr_{s,\cF}(G^F)$, there exists a unique $i\in \{1,\ldots,d\}$ 
such that $\langle \rho^*, \gamma_{u_i}\rangle=1$ 
and $\langle \rho^*, \gamma_{i'} \rangle=0$ for $i' \in \{1,\ldots,d\} 
\setminus \{i\}$. Thus, we have a partition $\Irr_{s,\cF}(G^F)=I_1 \amalg 
I_2 \amalg \cdots \amalg I_d$ such that
\[ \langle \rho^*,\gamma_{u_i}\rangle=\left\{ \begin{array}{cl} 1 & 
\qquad \mbox{if $\rho \in I_i$}, \\ 0 & \qquad \mbox{if $\rho\in I_j$ 
where $j\neq i$}.\end{array}\right.\]
Assume, if possible, that $I_r=\varnothing$ for some $r\in \{1,\ldots,d\}$. 
This means that $\langle\rho,\mbox{D}_G(\gamma_{u_r})\rangle=\langle 
\mbox{D}_G(\rho),\gamma_{u_r}\rangle=0$ for all $\rho\in \Irr_{s,\cF}(G^F)$.
Thus, by the definition of the scalar product, we have 
\[ 0=\frac{1}{|G^F|} \sum_{g \in G^F} \overline{\rho(g)}\, 
D_G(\gamma_{u_r})(g)\qquad \mbox{for all $\rho\in \Irr_{s,\cF}(G^F)$}.\]
Let $g \in G^F$ and assume that the corresponding term in the above sum is
non-zero. First of all, since $D_G(\gamma_{u_r})$ is unipotently supported,
$g$ must be unipotent. Let $\cO'$ be the conjugacy class of $g$. By 
\cite[6.13(i) and 8.6]{LuUS}, we have $D_G(\gamma_{u_r})(g)=0$ unless 
$\cO$ is contained in the closure of $\cO'$. Furthermore, by 
\cite[Theorem~11.2]{LuUS}, we have $\rho(g)=0$ unless $\cO'=\cO$ or 
$\dim \cO'<\dim \cO$. Hence, to evaluate the above sum, we only need to 
let $g$ run over all elements in $\cO^F$. Thus, we have 
\[  0=\sum_{j=1}^d \frac{1}{|\mbox{C}_{G^F}(u_j)|}\, \overline{\rho(u_j)}\, 
\mbox{D}_G(\gamma_{u_r}(u_j))\qquad \mbox{for all $\rho\in
\Irr_{s,\cF}(G^F)$}.\]
In particular, this holds for the characters $\rho_1,\ldots,\rho_d$ in
Corollary~\ref{cor35}. The invertibility of the matrix of values in
Corollary~\ref{cor35} then implies that $\mbox{D}_G(\gamma_{u_r})(u_j)=0$ 
for $1\leq j \leq d$. Thus, the restriction of $\mbox{D}_G(\gamma_{u_r})$ 
to $\cO^F$ is zero. Now, the relations in \cite[(2.4a)]{G1} (which are 
formally deduced from the main results in \cite{LuUS}) imply that $\langle 
\mbox{D}_G(\gamma_{u_r}),Y_{\cE}\rangle$ equals $\overline{Y_{\cE}(u_r)}$ 
times a non-zero scalar, for any $\cE \in \cI_\cO^F$. Hence, we have 
$Y_{\cE}(u_r)=0$ for any $\cE \in \cI_\cO^F$. However, this contradicts the 
fact that the matrix of values $\bigl(Y_{\cE}(u_j)\bigr)$ is invertible (see 
the remarks at the beginning of the proof of Corollary~\ref{cor35}). This 
contradiction shows that we have $I_i\neq \varnothing$ for all $i$. Now 
choose $\rho_i \in I_i$ for $1\leq i\leq d$. Then we have $\langle\rho_i^*,
\gamma_{u_j}\rangle=\delta_{ij}$ for $1\leq i,j\leq d$, as desired.
\end{proof}

\begin{rem} \label{rem42} In the setting of Proposition~\ref{prop41}, let
us drop the assumption that $\cG_\cF$ is abelian and assume instead that
$\cG_\cF$ is isomorphic to $\mathfrak{S}_3$, $\mathfrak{S}_4$ or 
$\mathfrak{S}_5$. (These cases occur when $G$ is simple modulo
its center and of exceptional type.) Then, by the Main Theorem~4.23 of 
\cite{L3}, we have a bijection $ \Irr_{s,\cF}(G^F)\leftrightarrow
\cM(\cG_\cF)$. 

Let $u_1,\ldots,u_d$ be representatives for the $G^F$-conjugacy classes 
contained in $\cO^F$. Since condition ($*$) in Proposition~\ref{prop24} is 
assumed to hold, we can identify $\cM(\cG_\cF)$ with the set of all pairs 
$(u_i,\sigma)$ where $1\leq i \leq d$ and $\sigma \in \Irr(\mbox{A}_G
(u_i)^F)$. Thus, via the above-mentioned bijection, we have a parametrization
\[ \Irr_{s,\cF}(G^F)=\{\rho_{(u_i,\sigma)} \mid 1\leq i \leq d,\,
\sigma \in \Irr(\mbox{A}_G(u_i)^F)\}.\]
On the other hand, Kawanaka \cite{K2}, \cite{K3} obtained explicit 
formulas for the values of the characters of the GGGR's (for $G$ of
exceptional type). Using these formulas, one can check that 
\[ \langle \rho_{u_i,\sigma}^*,\, \gamma_{u_j}\rangle= 
\left\{\begin{array}{cl} \sigma(1) & \qquad \mbox{if $i=j$},\\ 0 & 
\qquad\mbox{otherwise}.\end{array}\right.\]
Thus, setting $\rho_i:=\rho_{(u_i,1)}$ for $1\leq i \leq d$ (where $1$ 
stands for the trivial
character), we see that the conclusion of Proposition~\ref{prop41} holds
in these cases as well.
\end{rem}

\begin{thm} \label{KawConj} Recall our standing assumption that $p,q$ 
are large enough and the center of $G$ is connected. Then Kawanaka's 
Conjecture~\ref{Kconj} holds.
\end{thm}

\begin{proof} By standard reduction arguments, we can assume without
loss of generality that $G$ is simple modulo its center. If $G$ is of
type $A_n$ or of exceptional type, the assertion has been proved
by Kawanaka \cite[Theorem~2.4.3]{K3}, using his explicit formulas for
the character values of GGGR's. The following argument covers these
cases as well.

Let $\cO_1,\ldots,\cO_N$ be the $F$-stable unipotent classes of $G$,
where the numbering is chosen such that $\dim \cO_1 \leq \cdots \leq
\dim \cO_N$. By Proposition~\ref{prop25}, for each $i$, we can find 
an $F$-stable pair $(s_i,\cF_i) \in \cS_G$ such that $\cO_i=\Phi_G(s_i,
\cF_i)$ and condition ($*$) in Proposition~\ref{prop24} holds.

For each $i$, let $u_{i,1},\ldots,u_{i,d_i}$ be a set of representatives
for the $G^F$-conjugacy classes contained in $\cO_i^F$. Let
$\rho_{i,1},\ldots,\rho_{i,d_i}$ be irreducible characters as
in Proposition~\ref{prop41} (if $G$ is of classical type) or as in
Remark~\ref{rem42} (if $G$ is of exceptional type). We claim that 
\[ \langle \rho_{i_1,j_1}^*,\gamma_{u_{i_2,j_2}}\rangle=0
\qquad \mbox{if $i_1<i_2$}.\]
This is seen as follows. We have $\langle \rho_{i_1,j_1}^*,\gamma_{u_{i_2,
j_2}}\rangle=\pm \langle \rho_{i_1,j_1},\mbox{D}_G(\gamma_{u_{i_2,j_2}})
\rangle$. By the definition of the scalar product, we have
\[\langle \rho_{i_1,j_1},\mbox{D}_G(\gamma_{u_{i_2,j_2}})
\rangle=\frac{1}{|G^F|} \sum_{g \in G^F} \overline{\rho_{i_1,j_1}(g)}\,
\mbox{D}_G(\gamma_{u_{i_2,j_2}})(g).\] 
We now argue as in the proof of Proposition~\ref{prop41} to evaluate this 
sum. First of all, it's enough to let $g$ run over all unipotent elements of
$G^F$. Now let $g \in G^F$ be unipotent and assume, if possible, that the
corresponding term in the above sum is non-zero. The fact that $\rho_{i_1,
j_1}(g)\neq 0$ implies that the class of $g$ either equals $\cO_{i_1}$ or
has dimension $<\dim \cO_{i_1}$. Furthermore, the fact that $\mbox{D}_G
(\gamma_{u_{i_2,j_2}})(g)\neq 0$ implies that $\cO_{i_2}$ is contained in 
the closure of the class of $g$. Since we numbered the unipotent classes 
according to increasing dimension, we conclude that $\dim \cO_{i_1}=
\dim \cO_{i_2}$; furthermore, $g \in \cO_{i_1}$ and $\cO_{i_2}$ is 
contained in the closure of the class of $g$, which finally shows that
$\cO_{i_1}=\cO_{i_2}$, a contradiction. Thus, our assumption was wrong,
and the above scalar product is zero.

Together with the relations in Proposition~\ref{prop41} (or 
Remark~\ref{rem42}), we now see that the matrix of all scalar products
\[ \langle \rho_{i_1,j_2}^*, \gamma_{u_{i_2,j_2}}\rangle_{1\leq i_1,i_2
\leq N, \,1\leq j_1\leq d_{i_1},\,1\leq j_2\leq d_{i_2}}\]
is a block triangular matrix where each diagonal  block is an 
identity matrix. Hence that matrix of scalar products is invertible over
$\Z$ and so Kawanaka's conjecture holds by Lemma~\ref{triang1}.
\end{proof}

\begin{prop}[Characterisation of GGGR's] \label{Charac} Recall that $p,q$ 
are large enough and the center of $G$ is connected. Let $\cO$ be an 
$F$-stable unipotent class in $G$ and $\chi$ be a character of~$G^F$. 
Then $\chi=\gamma_u$ for some $u \in \cO^F$ if and only if the following 
three conditions are satisfied:
\begin{itemize}
\item[(a)] If $\chi(g) \neq 0$ for some $g \in G^F$, then the conjugacy 
class of $g$ is contained in the closure of $\cO$.
\item[(b)] If $\operatorname{D}_G(\chi)(g) \neq 0$ for some $g \in G^F$,
then $\cO$ is contained in the closure of the conjugacy class of $g$.
\item[(c)] We have $\chi(1)=|G^F|q^{-\dim \cO/2}$.
\end{itemize}
\end{prop}

\begin{proof} If $\chi=\gamma_u$ for some $u \in \cO^F$, then (a) and (c) 
are easily seen to hold by the construction of $\Gamma_u$; see 
Kawanaka \cite{K2}. Condition (b) is obtained as a consequence of 
\cite[6.13(i) and 8.6]{LuUS}. To prove the converse, by standard reduction 
arguments, we can assume without loss of generality that $G$ is simple 
modulo its center. Assume now that (a), (b) and (c) hold for $\chi$. 
Since $\chi$ is unipotently supported, we can write $\chi$ as an integral 
linear combination of the characters of the various GGGR's of $G^F$; see 
Theorem~\ref{KawConj}.

Now, given any $F$-stable unipotent class $\cO'$, the characters 
$\gamma_u$, where $u \in \cO'^F$, satisfy (a) with respect 
to $\cO'$. Hence, all characters $\gamma_u$, where $u$ is contained in 
the closure of $\cO$, satisfy (a). One easily deduces 
that {\em any} class function satisfying (a) is a linear combination 
of various $\gamma_u$ where $u$ is contained in the closure of
$\cO$. Similarly, {\em any} class function satisfying (b) is a linear 
combination of various $\mbox{D}_G(\gamma_u)$ where $\cO$ is contained
in the closure of the class of $u$. Hence, a class function satisfying 
both (a) and (b) will be a linear combination of various $\gamma_u$ 
such that $u \in \cO^F$. 

Let $u_1, \ldots,u_d$ be representatives for the $G^F$-conjugacy classes 
in $\cO^F$.  Then the above discussion shows that we can write $\chi=
\sum_{j=1}^d a_j \, \gamma_{u_j}$ where $a_j \in \Z$ for all $i$.

Now consider the characters $\rho_i$ in Proposition~\ref{prop41} (for
$G$ of classical type) or in Remark~\ref{rem42} (for $G$ of exceptional
type). Taking scalar products of $\chi$ with $\rho_i^*$, we find that 
$a_i \geq 0$ for all $i$ and so $\chi$ is a positive sum of characters 
of various GGGR's associated with $\cO^F$. All these GGGR's have dimension 
$|G^F|q^{-\dim \cO/2}$. Hence $\chi(1)$ is a positive integer multiple of 
$|G^F|q^{-\dim \cO/2}$. Condition (c) now forces that $\chi=\gamma_u$ for 
some $u \in \cO^F$, as required. 
\end{proof}

\end{document}